\documentclass[12pt]{amsart}
\usepackage[utf8]{inputenc}
\usepackage{amssymb}
\usepackage{amsmath}
\usepackage{amsfonts}
\usepackage{amsthm}
\usepackage{tikz}
\usepackage{algorithmic}
\usepackage{asymptote}
\usepackage{url}
\newtheorem{theorem}{Theorem}[section]
\newtheorem{corollary}{Corollary}[theorem]
\newtheorem{lemma}[theorem]{Lemma}

\newtheorem{conjecture}[theorem]{Conjecture}
\newtheorem{definition}[theorem]{Definition}

\usepackage{times}

\theoremstyle{definition}

\begin{document}

\title{Density of Numerical sets associated to a Numerical semigroup}

\author{Deepesh Singhal}
\address{Hong Kong University of Science and Technology, Hong Kong}
\email{dsinghal@connect.ust.hk}
\author{Yuxin Lin}
\address{University of Notre Dame, USA}
\email{ylin9@nd.edu}

\begin{abstract}
A numerical set is a co-finite subset of the natural numbers that contains zero. Its Frobenius number is the largest number in its complement. Each numerical set has an associated semigroup $A(T)=\{t\mid t+T\subseteq T\}$, which has the same Frobenius number as $T$. For a fixed Frobenius number $f$ there are $2^{f-1}$ numerical sets.
It is known that there is a number $\gamma$ close to $0.484$ such that the ratio of these numerical sets that are mapped to $N_f=\{0\}\cup(f,\infty)$ is asymptotically $\gamma$.
We identify a collection of families $N(D,f)$ of numerical semigroups such that for a fixed $D$ the ratio of the $2^{f-1}$ numerical sets that are mapped to $N(D,f)$ converges to a positive limit as $f$ goes to infinity.
We denote the limit as $\gamma_D$, these constants sum up to $1$ meaning that they asymptotically account for almost all numerical sets.
\end{abstract}

\maketitle


\section{Introduction}
$\mathbb{N}$ is the set of non-negative integers. A numerical set is a subset of $\mathbb{N}$ that contains zero and whose complement is finite. For example $\{0, 2,3,6 \rightarrow\}$ is a numerical set, the $\rightarrow$ indicates that every integer larger than 6 is in the numerical set. The largest number in the complement of a numerical set $T$ is called its Frobenius number, it is denoted by $f(T)$. For a fixed Frobenius number $f$ we have $2^{f-1}$ numerical sets. 

A numerical semigroup is a numerical set that is closed under addition. The exact number of numerical semigroups with a fixed Frobenius number is hard to evaluate, but \cite{Becklin} showed that it is of the order $2^{\overline{f}}$, where $\overline{f}=\lfloor\frac{f-1}{2}\rfloor$. Each numerical set has an associated semigroup defined as $A(T)=\{t\mid t+T\subseteq T\}$, where $t+T=\{t+x\mid x\in T\}$. It is easy to see that $A(T)$ is actually a numerical semigroup and it has the same Frobenius number as that of $T$. Given a numerical semigroup $S$ the number of numerical sets that have $S$ as their associated semigroup is denoted by $P(S)$.

If we consider all numerical semigroups with Frobenius number $f$, it is clear that
$$\sum_{f(S)=f}P(S)=2^{f-1}.$$
We define the density of a numerical semigroup $S$ (with $f(S)=f$) as $\mu(S)=\frac{P(S)}{2^{f-1}}$. The density of $S$ is the fraction of numerical sets with Frobenius number $f(S)$ that are mapped to $S$.
If we give a partial ordering by set theoretic containment on the set of numerical semigroups with Frobenius number $f$, it is observed that a few semigroups at the bottom have larger density and all of the other numerical semigroups have a very small density.
To illustrate this we consider numerical semigroups with Frobenius number $9$, there are $21$ of them. In the diagram below each semigroup is labeled by $S\cap[1,8]$, and its density is mentioned.

\includegraphics[width=0.9\textwidth]{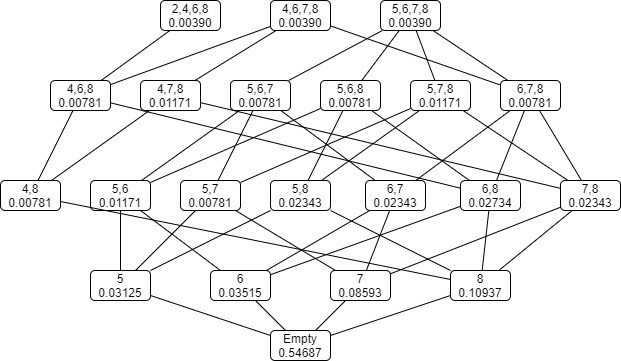}

Numbers in $S\cap[1,f(S)]$ are called small atoms of $S$. Note that for any $f$, the poset has a unique minimal semigroup $N_f=\{0,f+1\rightarrow\}$ and it has no small atoms.
In Marzuola and Miller's paper \cite{Miller} they compute $\mu(N_f)$ and prove that the limit
$\lim_{f\to\infty}\mu(N_f)$ exists and is positive.
This limit is denoted by $\gamma$ is approximately $0.484451$ with an error of at most $\pm 0.005011$. This means that a little less than half of the numerical sets are mapped to this numerical semigroup.

In \cite{Grobenius}, the authors study the density of numerical semigroups with one small atom and conjectured a similar result.
We denote numerical semigroups with one small atom as $N_{l,f}=\{0,f-l,f+1\rightarrow\}$, they conjectured that for any fixed $l$ the limit
$\lim_{f\to\infty}\mu(N_{l,f})$ exists.

We prove this conjecture and extend it to the case of $n$ small atoms.
Given a finite set $D$ of positive integers and a positive integer $f$ we define
$$N(D,f)=\{0\}\cup \{f-l\mid l\in D\} \cup \{f+1\rightarrow\}.$$
Note that if $f>\text{Max}(D)$ then $N(D,f)$ is a numerical set with Frobenius number $f$.
Moreover if $f>2Max(D)$ then all non zero elements of $N(D,f)$ are bigger than $\frac{f}{2}$, hence $N(D,f)$ is closed under addition and therefore is a numerical semigroup. It has $|D|$ small atoms.
Our main result is the following.
\begin{theorem}\label{lim n small atoms}
For a fixed finite subset of positive integers $D$ the following limit exists and is positive
$$\lim_{f\to\infty}\mu(N(D,f)).$$ 
\end{theorem}
It is proved in Theorem \ref{Main Result} and Theorem \ref{lims are positive}.
We denote the limit by $\gamma_{D}$. We list several of the values in Table \ref{t1}.

\begin{table}[h]
    \centering
    \begin{tabular}{|c|c|c|c|c|c|}
        \hline
        $D$ & $\gamma_{D}$ & $D$ & $\gamma_{D}$ & $D$ & $\gamma_{D}$\\
        \hline
        $\emptyset$ & $0.48660$ & 6 & $0.00700$ & 1,6 & $0.00186$  \\
        1 & $0.09476$ & 3,4 & $0.00443$ & 2,6 & $0.00174$ \\
        2 & $0.06079$ & 7 & $0.00435$ &1,2,3 & $0.00152$ \\
        3 & $0.02538$ & 2,5 & $0.00400$ &4,5 & $0.00132$ \\
        1,3 & $0.02035$ &  1,3,5 & $0.00332$ & 9 & $0.00131$  \\
        4 & $0.01793$ &  1,2,5 & $0.00280$ & 2,3,4 & $0.00106$ \\
        1,2 & $0.01683$ & 8 & $0.00269$ &1,2,6 & $0.00091$ \\
        2,3 & $0.01205$&  1,2,4 & $0.00228$& 1,3,4 & $0.00068$ \\
        1,4 & $0.01184$ &  1,5 & $0.00200$ &&\\
        5 & $0.01017$ &2,4 & $0.00191$ & &\\
        \hline
    \end{tabular}
    \caption{Values of limits with error at most $\pm 0.00212$}
    \label{t1}
\end{table}

We prove in Corollary \ref{Sum to 1} that
$$\sum_{|D|<\infty}\gamma_D=1.$$
This means that given $\epsilon>0$ one can find a finite collection of $D$s for which the sum of $\gamma_D$ is bigger than $1-\epsilon$. It follows that for sufficiently large $f$ the combined density of $N(D,f)$ for these $D$s is bigger than $1-\epsilon$ and in turn the combined density of all other numerical semigroups of Frobenius number $f$ is at most $\epsilon$.

In the following two theorems we investigate, given a general sequence of numerical semigroups when is their density positive in the limit. It is related to difference between the Frobenius number and the smallest positive element being bounded.
The smallest positive element is called the multiplicity and is denoted by $m(S)$. We define 
$R(S)=f(S)-m(S)$.

\begin{theorem}\label{positive liminf}
Given a sequence of numerical semigroups $S_n$, $R(S_n)$ is bounded if and only if
$$\liminf_{n\to\infty}\mu(S_n)>0.$$
\end{theorem}

\begin{theorem}
Given a sequence of numerical semigroups $S_n$, $R(S_n)$ goes to $\infty$ if and only if
$$\lim_{n\to\infty}\mu(S_n)=0.$$
\end{theorem}
\noindent These two theorems are proved in Theorem \ref{thm lim}, Corollary \ref{liminf converse}, Theorem \ref{thm liminf} and Corollary \ref{lim converse}.
Given any numerical semigroup $S$, we can uniquely express it as $N(D,f)$, indeed $D=D(S)=\{f-s\mid s\in S,1\leq s\leq f \}$ and $f=f(S)$.
It follows that $m(S)=f-Max(D)$ i.e. $Max(D)=R(S)$.
Notice that $R(S_n)$ is bounded if and only if the set $\{D(S_n)\mid n\in\mathbb{N}\}$ is finite i.e. the sequence $S_n$ is eventually a combination of sequences $N(D,f)$ for finitely many $D$.

\section{numerical semigroups with very small atoms}
A very small atom of $S$ is a positive element of $S$ which is smaller than $\frac{f(S)}{2}$.
In this section we show that the combined density of numerical semigroups which have very small atoms is small. This will be needed later.

\begin{theorem}\label{thm6}
For fixed $\epsilon>0$ (and $<1$) \[\lim_{f\to\infty}\frac{\#\{T\mid f(T)=f, m(A(T))\leq (1-\epsilon)f\}}{2^{f-1}}=0.\]
\end{theorem}
\begin{proof}
For the moment fix $m,f$ with $m<f$, $m \nmid f$. Say $f=mq+r$ with $0< r<m$. Consider the number of numerical sets $T$ for which $f(T)=f$ and $m(A(T))=m$, we will first prove that it is at most
$$(q+2)^{r-1}(q+1)^{m-r}.$$
Consider a numerical set with $A(T)=S$. Now if $x\in T$ then $x+m\in T$ as $m\in A(T)$. Consider the numbers $1,2,\dots,f-1$ and divide them in equivalence classes $mod\;m$. Note that $r-1$ of these equivalence classes have $q+1$ elements and $m-r$ equivalence classes have $q$ elements (we are not including the class $0(mod\;m)$). Now in each class we need to choose the first number to be included in $T$, or whether to include no number from that class. Thus the number of numerical sets is at most $(q+2)^{r-1}(q+1)^{m-r}$.

Next by the AM-GM inequality $$\left((q+2)^{r-1}(q+1)^{m-r}\right)^{\frac{1}{m-1}}
\leq\frac{(r-1)(q+2)+(m-r)(q+1)}{m-1}$$
$$=\frac{mq+r+m-q-2}{m-1}<\frac{f+m-1}{m-1}=1+\frac{f}{m-1} .$$
And thus the number of numerical sets with $m(A(T))=m$ and $f(T)=f$ is at most $\left(1+\frac{f}{m-1}\right)^{m-1}$.

Note that for a fixed $f$, $(1+\frac{f}{x})^{\frac{x}{f}}$ is an increasing function of $x$. Therefore by fixing $f$ and adding the contribution of each $m$ we get that
$$\frac{\#\{T\mid f(T)=f, m(A(T))\leq (1-\epsilon)f\}}{2^{f-1}}$$
$$\leq (1-\epsilon)f\frac{1}{2^{f-1}} \Big(1+\frac{1}{1-\epsilon}\Big)^{(1-\epsilon)f} <2f\Big(\frac{c}{2}\Big)^f$$
where $c=\Big(1+\frac{1}{1-\epsilon}\Big)^{(1-\epsilon)}<2$. Now letting $f\to\infty$ we are done.
\end{proof}

\begin{corollary}\label{less than f/2}
The size of the set $\{T\mid f(T)=f,m(A(T))\leq\frac{f}{2}\}$ is at most $f\times 3^{\frac{f}{2}}$.
\end{corollary}
\begin{proof}
If $\epsilon=\frac{1}{2}$ then $c=\sqrt{3}$.
\end{proof}

\section{The limits exist}
In this section we will prove part of Theorem \ref{lim n small atoms} by proving that the limit exists.
We start with a key observation that the last elements of $A(T)$ are determined by the last and the starting elements of $T$; the elements in the middle play no role.

\begin{lemma}\label{another definition}
If $f(T_1)=f_1$, $f(T_2)=f_2$, $T_1\cap[1,k]=T_2\cap[1,k]$ and for each $x\in [1,k]$ $f_1-x\in T_1$ if and only if $f_2-x\in T_2$, then for each $x\in [1,k]$ we have $f_1-x\in A(T_1)$ if and only if $f_2-x\in A(T_2)$.
\end{lemma}
\begin{proof}
Given $x\in [1,k]$, if $f_1-x\not \in A(T_1)$ then there exists $y\in T_1$ such that $f_1-x+y\not\in T_1$.
Therefore $f_1-x+y\leq f_1$ i.e. $y\leq x\leq k$ and hence $y\in T_2$ and $f_2-(x-y)\not\in T_2$. It follows that $f_2-x\not\in A(T_2)$.
\end{proof}

For example consider a numerical set $T$ with Frobenius number $f$. For $f-1$ to be in $A(T)$ firstly we need $f-1\in T$. Moreover for $x\geq 2$ $f-1+x\in T$, $f-1+1\not\in T$. Therefore $f-1\in A(T)$ if and only if $f-1$ is in $T$ and $1$ is not in $T$. Similarly the condition for having $f-2$ in $A(T)$ is that $f-2\in T$, $2\not\in T$ and if $1$ is in $T$ then so is $f-1$.

\begin{definition}
Given a finite set $D$ and positive integer $f>2Max(D)$ we define $B(D,f)$ to be the set of numerical sets $T$ with $f(T)=f$ and
$$A(T)\cap [f-Max(D),f-1]=\{f-l\mid l\in D\}.$$
Further we define $A_{D}=|B(D,2Max(D)+1)|$.
\end{definition}

\begin{lemma}\label{size of B(D,f)}
If $D$ is a finite set with $Max(D)=t$ and $f\geq 2t+1$ then
$$|B(D,f)|=A_{D}2^{f-2t-1}.$$
\end{lemma}
\begin{proof}
Consider a map $g$ from numerical sets with Frobenius number $f$ to those with Frobenius number $2t+1$, given by
$$g(T)
=\left(T\cap [1,t]\right)
\cup \left((T\cap [f-t,f-1])-(f-2t-1)\right)
\cup N_{2t+1}.$$
The map $g$ is surjective and every numerical set with Frobenius number $2t+1$ has exactly $2^{f-2t-1}$ numerical sets in its preimage.
Since $T\cap [1,t]=g(T)\cap [1,t]$ and for each $x \in [1,t]$ we have $f-x \in T$ if and only if $2t+1-x\in g(T)$, we can conclude by Lemma \ref{another definition} that for $y \in [1,t]$, $f-y \in A(T)$ if and only if $2t+1-y \in A(g(T))$. Thus $T \in B(D,f)$ if and only if $g(T) \in B(D,2t+1)$.
Therefore $|B(D,f)|=A_{D}2^{f-2t-1}$.
\end{proof}

We define $S(D,f)$ to be the subset of $B(D,f)$ that consists of all $T$ for which $m(A(T))\leq\frac{f}{2}$.
As proved in Corollary \ref{less than f/2}, $|S(D,f)|<f\times 3^{f/2}$, so it is relatively small compared to $2^f$.

\begin{theorem}\label{Main Corr}
Given a finite set $D$, $Max(D)=t$ and $f\geq 2t+1$
$$\mu(N(D,f))
=A_{D}4^{-t}-\sum_{k=t+1}^{\lfloor\frac{f-1}{2}\rfloor}A_{D\cup\{k\}}4^{-k}+O\Big(f\Big(\frac{\sqrt{3}}{2}\Big)^f\Big).$$
\end{theorem}
\begin{proof}
Call a numerical set $T$ good if $A(T)=N(D,f)$, we know that all good sets are contained in $B(D,f)$. 
Next, the complement of good numerical sets with respect to $B(D,f)$ consists of those $T$ for which $A(T)\cap [1,f-t-1]\neq\emptyset$.
If $A(T)\cap [1,f-t-1]\neq\emptyset$, pick the smallest $k$ for which $f-k\in A(T)\cap [1,f-t-1]$ and note that $T\in B(D\cup\{k\},f)$.
Therefore the complement can be written as a disjoint union
$\cup_{k=t+1}^{f-1}B(D\cup\{k\},f)$.
Moreover $$\bigcup_{k=\lceil\frac{f}{2}\rceil}^{f-1}B(D\cup\{k\},f)=S(D,f).$$
Therefore by Lemma \ref{size of B(D,f)}
$$P(N(D,f))
=|B(D,f)|-\sum_{k=t+1}^{\lfloor\frac{f-1}{2}\rfloor}|B(D\cup\{k\},f)|-|S(D,f)|$$
$$=A_{D}2^{f-2t-1}-\sum_{k=t+1}^{\lfloor\frac{f-1}{2}\rfloor}A_{D\cup\{k\}}2^{f-2k-1}-O\left(f\sqrt{3}^f\right).$$
\end{proof}

\begin{theorem}\label{Main Result}
Given a finite set $D$, $Max(D)=t$
\[\lim_{f \rightarrow \infty} \mu(N(D,f))=A_{D}4^{-t}-\sum_{k=t+1}^{\infty}A_{D\cup\{k\}}4^{-k}.\]
\end{theorem} 
\begin{proof}
Note that the right hand side of the equation is monotonically decreasing and bounded below by $0$ and hence convergent.
\end{proof}

\section{The limit is positive}
In this section we will show that the limits $\gamma_{D}$ are actually positive.

\begin{definition}
Given $k<f$, we define
$$G_l(f)=\{T\mid A(T)=N_f, T\cap [1,l]=\emptyset\} .$$
$$B_l(k,f)=\{T\mid T\in B(\{k\},f), T\cap [1,l]=\emptyset\} .$$
If $k\geq l$ we define $C_{l,k}$ to be the $|B_l(k,2k+1)|$.
And if $k\leq l$ then we define $C_{l,k}$ to be $|B_l(k,l+k+1)|$.
\end{definition}

We first prove that for a fixed $l$ the limit
$$\lim_{f\to\infty}\frac{|G_l(f)|}{2^{f-1}}$$
exists and is positive. The proof will follow a method similar to the last section.
This result will then be used to prove that $\gamma_D$ are all positive.

\begin{lemma}\label{B_l(k,f), k>l}
If $f\geq 2k+1$ and $k\geq l$ then $|B_l(k,f)|=C_{l,k}2^{f-2k-1}$.
\end{lemma}
\begin{proof}
Consider the map $g$ from numerical sets with Frobenius number $f$ to those with Frobenius number $2k+1$, given by
$$g(T)
=\left(T\cap [1,k]\right)
\cup \left((T\cap [f-k,f-1])-(f-2k-1)\right)
\cup N_{2k+1}.$$
The map $g$ is surjective and every numerical set with Frobenius number $2k+1$ has exactly $2^{f-2k-1}$ numerical sets in its preimage.
Lemma \ref{another definition} tells us that $T \in B(\{k\},f)$ if and only if $g(T) \in B(\{k\},2k+1)$.
Moreover since $l\leq k$ it is clear that $T\cap[1,l]=\emptyset$ if and only if $g(T)\cap[1,l]=\emptyset$.
Therefore $|B_l(k,f)|=C_{l,k}2^{f-2k-1}$.
\end{proof}

\begin{lemma}
If $f\geq l+k+1$ and $k<l$ then
$|B_l(k,f)|=C_{l,k}2^{f-k-l-1}$.
\end{lemma}
\begin{proof}
Consider the map $g$ from numerical sets with Frobenius number $f$ to those with Frobenius number $l+k+1$, given by
$$g(T)
=\left(T\cap [1,l]\right)
\cup \left((T\cap [f-k,f-1])-(f-l-k-1)\right)
\cup N_{l+k+1}.$$
The map $g$ is surjective.
and every numerical set with Frobenius number $l+k+1$ has exactly $2^{f-l-k-1}$ numerical sets in its preimage.
Lemma \ref{another definition} tells us that
$T \in B(\{k\},f)$ if and only if $g(T) \in B(\{k\},l+k+1)$.
Also $T\cap[1,l]=\emptyset$ if and only if $g(T)\cap[1,l]=\emptyset$.
Therefore $|B_l(k,f)|=C_{l,k}2^{f-l-k-1}$.
\end{proof}

Note that $f>2l$, $k<\frac{f}{2}$ will ensure that $f-k>l$ i.e. $f\geq k+l+1$.

\begin{theorem}
Fix $l$, then  $\lim_{f\to\infty}\frac{|G_l(f)|}{2^{f-1}}$ exists and is positive.
\end{theorem}
\begin{proof}
Assume $f>2l$,
we start with the collection of all numerical sets $T$ with Frobenius number $f$ for which $T\cap [1,l]=\emptyset$.
There are $2^{f-1-l}$ of them.
The complement of $G_l(f)$ consists of those $T$ for which $A(T)\cap [1,f-1]\neq\emptyset$.
Given such a $T$, it belongs to a unique $B_l(k,f)$.
The $k$ is determined such that $f-k$ is the largest element of $A(T)\cap [1,f-1]$.

Now $2(f-k)\in A(T)$ as $A(T)$ is additively closed, and $2(f-k)>f-k$ therefore $2(f-k)>f$ i.e. $k<\frac{f}{2}$.
Therefore the complement is the disjoint union of $B_l(k,f)$ as $k$ ranges from $1$ to $\lfloor\frac{f-1}{2}\rfloor$.

$$|G_l(f)|=2^{f-1-l}-\sum_{k=1}^{\lfloor\frac{f-1}{2}\rfloor}|B_l(k,f)|.$$

Therefore the limit is
$$\lim_{f\to\infty}\frac{G_l(f)}{2^{f-1}}
=\frac{1}{2^l}-\sum_{k=1}^{l}C_{l,k}2^{-l-k}-\sum_{k=l+1}^{\infty}C_{l,k}4^{-k} . $$

Next we estimate the constants $C_{l,k}$.
First consider the case $k\leq l$, we will prove $C_{l,k}=1$. Let $f=l+k+1$ and $T\in B_l(k,f)$. We know $f-k=l+1\in T$, $T\cap[1,l]=\emptyset$.
If some $x\in[f-k+1,f-1]$ is in $T$, then for any non-zero $y$ in $T$ $x+y\geq f-k+1+l+1>f$ which would imply that $x\in A(T)$.
Therefore $x$ cannot be in $T$, it follows that $T=\{f-k\}\cup N_{f}$ and $C_{l,k}=1$.

Second if $l<k\leq 2l+1$, we will again prove $C_{l,k}=1$. Let $f=2k+1$, $T\in B_l(k,f)$.
By the argument in the first case it follows that $T\cap [f-l,f-1]=\emptyset$.
Since $f-k\in A(T)$ it follows that $T\cap [k-l,k-1]=\emptyset$.
Now since $k-l\leq l+1$ we have $T\cap [1,k-1]=\emptyset$ and hence $T\cap [f-k+1,f-1]=\emptyset$.
Therefore $T=\{f-k\}\cup N_{f}$.

Lastly for $k\geq 2l+2$ we will show that $C_{l,k}\leq 2^l\times 3^{k-2l-1}$.
Let $f=2k+1$ and consider a $T$ in $B_l(k,f)$.
$T\cap [1,l]=\emptyset$, this implies that $T\cap[f-l,f-1]=\emptyset$. Since $f-k\in A(T)$ this in turn implies that $T\cap [k-l,k-1]=\emptyset$. We see that $T$ is a subset of
$$[l+1,k-l-1]\cup \{k,k+1\}\cup [k+2,k+l+1]\cup [k+l+2,2k-l]\cup N_f.$$
We know that $k+1=f-k\in T$, $k\not\in T$ since $f-k\in A(T)$.
Moreover $k+1\in A(T)$ also implies that whenever some $x\in [l+1,k-l-1]$ in in $T$ then $k+1+x\in [k+l+2,2k-l]$ is also in $T$.
Counting numerical sets with these conditions we have $3^{k-2l-1}$ choices for $[l+1,k-l-1]\cup [k+l+2,2k-l]$ and $2^l$ choices for $[k+2,k+l+1]$.
In conclusion $C_{l,k}\leq 2^l3^{k-2l-1}$.

We now see that
$$\sum_{k=1}^{l}C_{l,k}2^{-l-k}=2^{-l}\sum_{k=1}^{l}\frac{1}{2^k}=\frac{1}{2^l}-\frac{1}{4^l},$$
$$\sum_{k=l+1}^{2l+1}C_{l,k}4^{-k}=\sum_{k=l+1}^{2l+1}\frac{1}{4^{k}}=\frac{1}{3}\frac{1}{4^l}-\frac{1}{3}\frac{1}{4^{2l+1}} . $$

Therefore $$\frac{1}{2^l}-\sum_{k=1}^{l}C_{l,k}2^{-l-k}-\sum_{k=l+1}^{2l+1}C_{l,k}4^{-k}=\frac{2}{3}\frac{1}{4^l}+\frac{1}{3}\frac{1}{4^{2l+1}}>\frac{2}{3}\frac{1}{4^l} . $$

Also $$\sum_{k=2l+2}^{\infty}C_{l,k}4^{-k}\leq \sum_{k=2l+2}^{\infty}2^l\times 3^{k-2l-1}\times 4^{-k}=\frac{2^l\times 3}{4^{2l+1}} . $$

Finally note that $$\frac{2}{3}\frac{1}{4^l}>\frac{3\times 2^l}{4^{2l+1}}$$ and hence the limit is positive.
\end{proof}

Denote the constant $a_l=\frac{2}{3}\frac{1}{4^l}-\frac{3\times 2^l}{4^{2l+1}}$.

\begin{theorem}\label{lims are positive}
Given a finite $D$ with $Max(D)=t$ we have $\gamma_{D}\geq \frac{a_{t}}{2^{t+1}}$, and in particular it is positive.
\end{theorem}
\begin{proof}
We create an injective map from $G_{t}(f-t-1)$ to the numerical sets associated with $N(D,f)$.
Given $T_1\in G_{t}(f-t-1)$ let 
$$T=(T_1\cap [0,f-t-1])\cup \{f-l\mid l\in D\}\cup\{f+1\rightarrow\} . $$
Note that $f(T)=f$ and $T\cap [1,t]=\emptyset$, therefore 
$$A(T)\cap [f-t,f]=\{f-l\mid l\in D\} .$$ 
Given $1\leq x<f-t$ if $x\in T$ then $x\in T_1$ but $x\not\in A(T_1)$ so there exists $y\in T_1$ such that $x+y\not\in T_1$. In this case we must have $x+y\leq f-t-1$ and hence $y\in T$, $x+y\not\in T$.
Therefore $A(T)=N(D,f)$.

Finally the size of $G_{t}(f-t-1)$ is at least $a_{t}2^{f-t-2}(1+o(1))$ and the result follows.
\end{proof}

We next investigate how fast the sequence for $\gamma_D$ converges.
\begin{lemma}\label{estimate A}
For a finite set $D$ with $Max(D)=t$ we have $A_{D}\leq 3^{t-1}$.
\end{lemma}
\begin{proof}
Let $f=2t+1$, consider a $T$ in $B(D,f)$. We know that $t+1=f-t\in A(T)$ so $t+1\in T$ and $t\not\in T$.
Also whenever some $x\in [1,t-1]$ is in $T$ then $t+1+x\in [t+2,2t]$ is also in $T$.
Therefore $A_{D}\leq 3^{t-1}$.
\end{proof}

It follows that $\sum_{k=N+1}^{\infty}A_{D\cup\{k\}}4^{-k}
\leq (\frac{3}{4})^N$. This was used to compute Table \ref{t1}.

\section{General sequences of numerical semigroups}

Recall that for a numerical semigroup $S$, $R(S)=f(S)-m(S)$.
We have obtained families $N(D,f)$ such that each family has a positive density in the limit.
We also know that $R(N(D,f))=Max(D)$ remains constant within each family.
One might wonder if there are other families of numerical semigroups which asymptotically have a positive density. We show that this is essentially determined with $R(S)$ being bounded. Moreover any sequence of numerical semigroups for which $R(S)$ is bounded is obtained by combining finitely many $N(D,f)$ sequences.

\begin{definition}
We define
$$\alpha_n(f)=\frac{\#\{T\mid f(T)=f, R(A(T))=n\}}{2^{f-1}} .$$

Also let
$\alpha_n=\lim_{f\to\infty}\alpha_n(f)$.
\end{definition}

Theorem \ref{Main Result} implies that these limits actually exist, for example $\alpha_{-1}=\gamma$, $\alpha_{1}=\gamma_1$, $\alpha_{2}=\gamma_{2}+\gamma_{1,2}$, $\alpha_3=\gamma_{3}+\gamma_{1,3}+\gamma_{2,3}+\gamma_{1,2,3}$. In general $\alpha_n$ is the sum of $2^{n-1}$ terms, one for each subset of $\{1,2,\dots,n-1\}$.

We use the notation $\overline{f}=\lfloor\frac{f-1}{2}\rfloor$.

\begin{theorem}
$\sum_{n=-1}^{\infty}\alpha_n=1$.
\end{theorem}
\begin{proof}
Fix a $M$ for which $2^{\frac{1}{M}}\sqrt{3}<2$.
Note that $\sum_{n=-1}^{f-1}\alpha_n(f)=1$. Now Theorem $\ref{thm6}$ with $\epsilon = 1-\frac{1}{M}$ implies that 
$$\sum_{n=\frac{f}{M}}^{f-1}\alpha_n(f)=\frac{\#\{T\mid f(T)=f, m(A(T))\leq (1-\epsilon)f\}}{2^{f-1}} $$
goes to 0 as $f$ goes to infinity. Thus $\sum_{n=\frac{f}{M}}^{f-1}\alpha_n(f)=o(1)$.
And $\sum_{n=-1}^{\frac{f}{M}}\alpha_n(f)=1-o(1)$.

Consider some $n<\frac{f}{2}$,
by Theorem \ref{Main Result} and Theorem \ref{Main Corr} we know that
$$\alpha_n(f)-\alpha_n
=\sum_{D\subseteq\{1,2,\dots,n-1\}}
\sum_{k=\overline{f}+1}^{\infty}
A_{D\cup\{n,k\}}4^{-k}+O\Big(f\Big(\frac{\sqrt{3}}{2}\Big)^f\Big).$$
We know by Lemma \ref{estimate A} that $A_{D\cup\{n,k\}}\leq 3^{k-1}$.
Therefore 
$$|\alpha_n(f)-\alpha_n|\leq 2^{n-1}\sum_{k=\overline{f}+1}^{\infty}\frac{1}{3}\Big(\frac{3}{4}\Big)^k+O\Big(f\Big(\frac{\sqrt{3}}{2}\Big)^f\Big)$$
$$=2^{n-1}\Big(\frac{3}{4}\Big)^{\overline{f}}+ O\Big(f\Big(\frac{\sqrt{3}}{2}\Big)^f\Big).$$

Next we have
$$\sum_{n=1}^{\frac{f}{M}}|\alpha_n(f)-\alpha_n|
\leq \sum_{n=1}^{\frac{f}{M}}2^{n-1}\Big(\frac{3}{4}\Big)^{\overline{f}}+ O\Big(\frac{f^2}{M}\Big(\frac{\sqrt{3}}{2}\Big)^f\Big)$$

$$<2^{\frac{f}{M}}\Big(\frac{3}{4}\Big)^{\frac{f}{2}}+O\Big(f^2\Big(\frac{\sqrt{3}}{2}\Big)^f\Big)
=\Big(\frac{2^{\frac{1}{M}}\sqrt{3}}{2}\Big)^f+O\Big(f^2\Big(\frac{\sqrt{3}}{2}\Big)^f\Big)=o(1).$$

Combining the two it follows that
$$\sum_{n=-1}^{\frac{f}{M}}\alpha_n
=\sum_{n=-1}^{\frac{f}{M}}\alpha_n(f)
-\sum_{n=1}^{\frac{f}{M}}(\alpha_n(f)-\alpha_n)
-(\alpha_{-1}(f)-\alpha_{-1})$$
$$=1-o(1)-o(1)-o(1)=1-o(1).$$ 
Letting $f$ tend to infinity we are done.
\end{proof}

\begin{corollary}\label{Sum to 1}
$\sum_{|D|<\infty}\gamma_D=1$, where $D$ ranges over all finite subsets of positive integers.
\end{corollary}

\begin{theorem}\label{thm lim}
If $S_n$ is a sequence of numerical semigroups such that\\ $\lim_{n\to\infty}R(S_n)=\infty$ then
$$\lim_{n\to\infty}\mu(S_n)=0.$$
\end{theorem}
\begin{proof}
Since $R(S_n)\to\infty$ we have $f(S_n)\to\infty$. Given $\epsilon>0$ there is a $N$ such that $1-\sum_{i=-1}^{N}\alpha_i<\epsilon$. Also there is a $N'$ (depending on $N$) such that $n>N'$  implies $\sum_{i=-1}^{N}|\alpha_i(f(S_n))-\alpha_i|<\epsilon$. Therefore $n>N'$ implies 
$$1-\sum_{i=-1}^{N}\alpha_i(f(S_n))<2\epsilon.$$ 
Also there is a $N''$ such that $n>N''$ implies $R(S_n)>N$.
Then for $n>Max (N',N'')$ we have
$$\mu(S_n)
\leq\alpha_{R(S_n)}(f(S_n))
<1-\sum_{i=-1}^{N}\alpha_{i}(f(S_n))
<2\epsilon.$$
Therefore
$$\limsup_{n\to\infty}\mu(S_n)
\leq 2\epsilon.$$
Since $\epsilon$ was arbitrary we are done.
\end{proof}
\begin{corollary}\label{liminf converse}
If $S_n$ is a sequence of numerical semigroups for which  $$\liminf_{n\to\infty}\mu(S_n)>0$$
then $R(S_n)$ is bounded.
\end{corollary}
\begin{proof}
If $R(S_n)$ is not bounded, then there is a subsequence for which $R(S)$ goes to $\infty$. Now Theorem \ref{thm lim} leads to a contradiction.
\end{proof}

\begin{theorem}\label{thm liminf}
If $S_n$ is a sequence of numerical semigroups for which $R(S_n)$ is bounded above then
$$\liminf_{n\to\infty}\mu(S_n)>0.$$
\end{theorem}
Proof: Say $R(S_n)<M$ for each $n$. Each semigroup can be written as $N(D,f)$ for some $D$ with $Max(D)<M$. There are $2^{M-1}$ such $D$s. Create a separate sub-sequence for each choice of $D$. Now within a sub-sequence remember that $$\lim_{f\to\infty}\mu(N(D,f))=\gamma_{D} .$$ which we have proven to be positive. Therefore each sub-sequence has a positive liminf and hence the original sequence has a positive liminf.

\begin{corollary}\label{lim converse}
If $S_n$ is a sequence of semigroups such that
$$\lim_{n\to\infty}\mu(S_n)=0$$
then $\lim_{n\to\infty}R(S_n)=\infty$.
\end{corollary}
\begin{proof}
If $R(S_n)$ does not go to infinity then there will be a subsequence for which $R(S_n)$ is bounded. Then Theorem \ref{thm liminf} leads to a contradiction.
\end{proof}

\begin{conjecture}
If $D_1\neq D_2$ then $\gamma_{D_1}\neq\gamma_{D_2}$.
\end{conjecture}
A consequence of this conjecture would be the following
\begin{corollary}
If $S_n$ is a sequence of distinct numerical semigroups then the sequence $\mu(S_n)$ converges if and only if either $R(S_n)\to\infty$ or there exist $D$ such that $S_n$ is eventually a sub-sequence of $N(D,f)$.
\end{corollary}

\section{Acknowledgements}
Thanks to San Diego State University REU 2019 for introducing us to some of the questions investigated in this paper. Thanks to Dr Chris O'Neill and Dr David Galvin for giving us useful comments on it.


\begin{thebibliography}{}


\bibitem{Grobenius}
Arroyo, J., Autry, J., Crandall, C., Lefler, J., Ponomarenko, V. (2020).
On numerical semigroups with almost-maximal genus.
\emph{The PUMP Journal of Undergraduate Research,}
3, pp 62-67.

\bibitem{Becklin}
Backelin J., (1990).
On the number of semigroups of natural numbers,
\emph{Mathematica Scandinavica,}
66, pp 197-215. 
https://doi.org/10.7146/math.scand.a-12304

\bibitem{Void Poset}
Chen A., O'Neill C., Singhal D., Lawson L.,
Enumerating the associated numerical sets of a numerical semigroup.
Unpublished manuscript.

\bibitem{}
Constantin H., Houston-Edwards B., Kaplan N. (2017).
Numerical sets, core partitions, and integer points in polytopes.
In: \emph{Nathanson M. (eds) Combinatorial and Additive Number Theory II.} CANT 2015, CANT 2016. Springer Proceedings in Mathematics and Statistics, vol 220. Springer, Cham.

\bibitem{Miller}
Marzuolaa J., Miller A., (2010).
Counting numerical sets with no small atoms,
\emph{Journal of Combinatorial Theory,}
Series A, Volume 117, Issue 6,
pp 650-667,
ISSN 0097-3165,
https://doi.org/10.1016/j.jcta.2010.03.002.

\bibitem{Sanches Textbook}
Rosales J.C., Garcia-Sanchez P.A., (2009).
\emph{Numerical Semigroups}
New York, NY, USA: Springer.


\end{thebibliography}
\end{document}